\documentclass[11pt,a4paper]{article}

\usepackage{authblk}
\usepackage{graphicx}
\usepackage{multicol,multirow}
\usepackage{amsmath,amssymb,amsfonts}
\usepackage{mathrsfs}
\usepackage{rotating}
\usepackage{appendix}
\usepackage[numbers]{natbib}
\usepackage{amsthm}
\usepackage{lipsum}
\usepackage{hyperref}

\usepackage{pb-diagram}
\usepackage{tikz-cd}

\usepackage{caption}
\theoremstyle{cupplain}
\newtheorem{theorem}{Theorem}[section]
\newtheorem{lemma}[theorem]{Lemma}
\newtheorem{prop}[theorem]{Proposition}
\newtheorem{corollary}[theorem]{Corollary}
\theoremstyle{cupdefinition}
\newtheorem{definition}{Definition}[section]
\theoremstyle{cupremark}

\newtheorem{example}[theorem]{Example}
\theoremstyle{cupproof}

\numberwithin{equation}{section}

\renewcommand{\r}[1]{\text{\textsc{Rel}}\left({#1}\right)}

\newcommand{\mto}{\multimap}
\newcommand{\fr}[1]{\text{\textsc{FRel}}\left({#1}\right)}
\newcommand{\iso}[1]{\overset{\scriptscriptstyle #1}{\simeq}}
\newcommand{\apply}[3]{{#2}\,{#1}\,{#3}}

\newcommand{\N}[0]{\mathbb{N}}
\newcommand{\Z}[0]{\mathbb{Z}}
\def\mathobj#1{\mbox{$#1$}}
\def\NN{\mathobj{\mathbb{N}}}

\newcommand{\C}[0]{\mathcal{C}}
\newcommand{\D}[0]{\mathcal{D}}
\def\cC{\text{$\mathcal C$}}
\def\cD{\text{$\mathcal D$}}

\renewcommand{\a}[0]{\alpha}
\renewcommand{\b}[0]{\beta}
\renewcommand{\c}[0]{\gamma}
\renewcommand{\phi}[0]{\varphi}
\renewcommand{\i}[0]{\iota}
\renewcommand{\l}[0]{\lambda}
\newcommand{\m}[0]{\mu}

\newcommand{\Endo}[0]{\operatorname{End}}
\newcommand{\Szym}[0]{\operatorname{Szym}}
\newcommand{\LE}[0]{\operatorname{LE}}
\newcommand{\LM}[0]{\operatorname{LM}}
\newcommand{\gker}{\operatorname{gker}}
\newcommand{\gim}{\operatorname{gim}}
\newcommand{\id}{\operatorname{id}}
\newcommand{\card}{\operatorname{card}}

\newcommand{\Feq}[1]{\simeq_{#1}}

\newcommand{\Leq}[0]{\Feq{L}}
\newcommand{\generatedby}[1]{\left\langle{#1}\right\rangle}

\providecommand{\keywords}[1]
{
  \small	
  \noindent{\textit{Keywords: }} #1
}

\providecommand{\msc}[1]
{
  \small	
  \noindent{\textit{MSC2020: }} #1
}

\begin{document}

\title{Linear Relations of Finite Length Modules are Shift Equivalent to Maps\thanks{The research of MP was partially supported by Polish National Science Center under Opus Grant 2019/35/B/ST1/00874.

The publication has been supported by a grant from the Faculty of Mathematics and Computer Science under the Strategic Programme Excellence Initiative at Jagiellonian University.

We gratefully acknowledge Polish high-performance computing infrastructure PLGrid (HPC Center: ACK Cyfronet AGH) for providing computer facilities and support within computational grant no. PLG/2023/016390.}}

\author[1]{Bartosz Furmanek}
\author[1]{Filip Oskar {\L}anecki}
\author[1]{Mateusz Przybylski}
\author[2]{Jim Wiseman}

\affil[1]{Faculty of Mathematics and Computer Science, Jagiellonian University, {\L}ojasiewicza 6, Krak\'ow, 30-348, Poland\\
(\it{mateusz.przybylski@uj.edu.pl})}
\affil[2]{Department of Mathematics, Agnes Scott College, 141 East College Avenue, Decatur, 30030, GA,  USA}

\maketitle

\begin{abstract}
Linear relations, defined as submodules of the direct sum of two modules, can be viewed as objects that carry dynamical information and reflect the inherent uncertainty of sampled dynamics. These objects also provide an algebraic structure that enables the definition of subtle invariants for dynamical systems. In this paper, we prove that linear relations defined on modules of finite length are shift equivalent to bijective mappings. 
\end{abstract}

\keywords{linear relations, shift equivalence, Szymczak category, Leray functor}

\msc{{37B30} (Primary), {18B10} (Secondary)}

%

\section{Introduction}

A multivalued map is a convenient tool for analyzing sampled dynamics, that is the dynamics derived from experiments, various types of data such as time series, and similar sources. By its construction, a multivalued map can associate an argument with a set of values, which can be utilized to capture uncertainties in the data arising from measurements, reduce the volume of data, and more. Such a multivalued map can then serve as a generator of a discrete dynamical system. 

To investigate dynamical phenomena present in the data, various methods can be employed, such as machine learning approaches \cite{J2020}, topological invariants \cite{BMMP2020}, and others \cite{BGHKMV2024}. Since topological invariants appear to have a number of advantages while dealing with data, this approach is of interest. One of the particular methods to apply is Conley index theory \cite{Con1978}. Initially, the index was developed for flows. Later, the method was extended to discrete dynamical systems \cite{RS1988,Mr1990,Sz1995} and for their multivalued counterparts \cite{KM1995,BMW2024}, allowing its usage in applications and sampled dynamics \cite{BMMP2020,BGHKMV2024} among others. 

In the continuous case, the Conley index is a homotopical invariant, while in the discrete time case it is defined in an algebraic setting. Since a map $f\colon X\to X$ cannot be directly represented in a computational setting, it is necessary to use enclosures of the map $f$, which may take a form of a multivalued map $F\colon X\mto X$. Different techniques, such as interval arithmetic \cite{M1966} or binning \cite{BMMP2020}, may be used to obtain the enclosure of $f$, that is a multivalued map $F$, such that $f(x)\in F(X)$ for $x\in X$. When $F$ satisfies certain additional conditions, the multivalued map may be used to compute $f_*$, the map induced in homology by $f$. This $f_*$ map is a representative of some class of equivalent objects. The class is the Conley index. 

Typically, in order to obtain $f_*$, one has to find the projections $p$ and $q$ from the graph of multivalued map $F$ to the domain and the codomain of $F$, respectively. If fibers of $p$ are acyclic, then $p_*$, the map induced in homology by $p$, is an isomorphism by the Vietoris-Begle Theorem \cite{V1927}, and $f_*=q_*\circ p_*^{-1}$. The acyclicity condition is equivalent to the acyclicity of values of $F$. However, this acyclicity condition is the most challenging condition imposed on the map $F$ \cite{BMMP2020}. To overcome this challenge, one often has to artificially enlarge the enclosures, which leads to overestimation of values and, consequently, a loss of dynamical subtlety carried by the multivalued map $F$. 

In this context, our aim is to obtain Conley-type information without necessity of enlarging the values of multivalued maps. To achieve this, we need to address the following issue: the construction of the Conley index mentioned earlier cannot proceed without assuming the acyclicity of the values of $F$. Therefore, we need to investigate an earlier step - specifically, the Szymczak category. Szymczak proved \cite{Sz1995} that in the context of Conley theory for discrete dynamical systems, there exists a universal functorial construction that, as a result, allows the identification of all objects carrying similar dynamical information. In other words, any other construction \cite{RS1988,Mr1990} factorizes through the functor and category proposed by Szymczak. Fortunately, his construction is general and may be applied to any category. Initial attempts are proposed in \cite{PMW2023,AMPW2024}, where the authors study classes of isomorphic objects in the Szymczak category within the context of the category of finite sets and relations. 

Recall that the class of isomorphic objects in the Szymczak category coincides with the class of shift equivalent objects (cf. \cite{FR2000}). Despite the fact that the concept of shift equivalence is more widely recognized, we prefer to remain within the framework of category theory, as this approach allows us to capture certain subtleties associated with the construction of the Leray functor \cite{Mr1990}. The functor is particularly convenient in applications due to its algebraic conciseness, which facilitates computational implementation. 

Returning to the issue of the necessity of enlarging the values of multivalued maps, we propose to study the situation of non-acyclic values of multivalued maps. In this case, the map induced in homology by $F$ is no longer a map but a linear relation (also called an additive relation, see \cite{ML1995}, or a multivalued linear operator, see \cite{CC2017}). Recall that for $R$-modules $A$ and $B$ over some commutative ring with unity $R$, a subset $\alpha \subseteq A\oplus B$ is a linear relation if it is a submodule of $A\oplus B$. Therefore, a linear relation may be seen as a binary relation with an additional algebraic structure given by module operation, or as a multivalued and partial homomorphism of modules. The class of all modules as objects along with all linear relations as morphisms, with the composition of relations as usual, forms a category \cite{ML1995}. 

Now, by combining these two perspectives, we can achieve the most general input for potential Conley theory without the limitations imposed by acyclic values. However, a general classification of classes of isomorphic objects in the Szymczak category for linear relations is daunting. The main difficulty lies in the multitude of different specialized structures between modules. Moreover, the classification obtained in \cite{PMW2023,AMPW2024} cannot be directly applied to the context of the Szymczak category, even for linear relations on finite modules, because this category is not a subcategory of finite sets and relations. 

Therefore, our approach is as follows. We begin with the simplest example: linear relations on finite one dimensional modules over finite fields. We present the example and the resulting statements in Section \ref{sec:zp_fields}. Next, we extend our discussion to the general case, namely linear relations on modules of finite length. In Section \ref{sec:general_case} we prove that there exists a map in every class of isomorphic objects in the Szymczak category for the full subcategory of linear relations on modules of finite length. Furthermore, this map is bijective and uniquely determined. In other words, the linear relation is shift equivalent to a bijection. Notably, we observe that these maps are obtained by applying the Leray functor \cite{Mr1990}, a convenient normal functor used prior to Szymczak’s development of the universal functor. Based on these observations, we combine the results to prove the main theorems at the end of Section \ref{sec:equivalence}, summarizing our findings. In the Appendix, we elaborate on an example that highlights the necessity of assuming the finite length of considered modules.

The question of a full classification of isomorphic objects in the Szymczak category for linear relations on arbitrary modules remains open.

%

\section{Main results}

Let $R$ be a commutative ring with unity. Modules of finite length as objects, along with linear relations as morphisms with the standard composition of relations, constitute a category $\fr{R}$ (see Sections \ref{sec:linear_relations} and \ref{sec:equivalence} for details). We consider the category of endomorphisms over $\fr{R}$, that is a category whose objects are pairs $(A,\a)$, where $A$ is an object and $\a$ is a morphism of $\fr{R}$ (see Section \ref{sec:categories_and_functors} for details). The Leray functor $L\colon\Endo\fr{R}\to\Endo\fr{R}$ is an endofunctor. Moreover, $L$ maps every endomorphism in $\fr{R}$ to an isomorphism, which is a bijective mapping (see Theorem \ref{thm:Leray is normal}). 

The category $\Szym\fr{R}$ has the same objects as $\Endo\fr{R}$, but its morphisms are defined in a more sophisticated way (see Section \ref{sec:categories_and_functors}). As mentioned in the Introduction, two isomorphic objects in the Szymczak category are shift equivalent \cite{FR2000}. Thus, the Szymczak category is the category where the concept of shift equivalence coincides with isomorphism. 

Our first main theorem states that in the class of isomorphic objects in the Szymczak category for $\fr{R}$ (equivalently, in the class of shift equivalent objects) there exists an object equipped with a bijective mapping.

\begin{theorem}
    \label{thm:Special Szymczak representative}
    $(A,\a)\iso{\Szym} L(A,\a)$ for every $(A,a)\in\Endo\fr{R}$.
\end{theorem}

The next two theorems assert that in the context of the category of linear relations on modules of finite length, the Leray functor and the Szymczak functor are equivalent. 

\begin{theorem}
    \label{thm:Leray and Szymczak equivalence}
    Let $(A,a),\ (B,\b)\in\Endo\fr{R}$. Then $(A,\a)\iso{\Szym} (B,\b)$ if and only if $L(A,\a)\iso{\Endo} L(B,\b)$
\end{theorem}

\begin{theorem}
    \label{thm: Szymczak on auto C is identity functor}
    Let $(A,a),\ (B,\b)\in\Endo\fr{R}$ be such that $\a$ and $\b$ are bijections. If $(A,\a)\iso{\Szym} (B,\b)$ then $(A,\a)\iso{\Endo} (B,\b)$.
\end{theorem}

Thus every class of objects isomorphic under the Szymczak functor contains a canonical object $L(A, \a)$, which is unique up to isomorphism in $\Endo \fr{R}$, where $(A,\a)$ is any object in the class.

%

\section{Linear relations}
\label{sec:linear_relations}

Let $M,N$ be $R$-modules. A subset $\phi\subseteq M\oplus N$ is called a \emph{linear relation} if it is a submodule of $M\oplus N$. 
To emphasize the fact that one may think of $\phi$ as a partial multivalued map,
we will also write $n\in\phi(m)$ or $\apply{\phi}{n}{m}$ as an alternative to $(m,n)\in\phi$
and denote $\phi\colon M\mto N$ instead of $\phi\subseteq M\oplus N$.
Indeed, for $m\in M$ we understand 
\begin{align}
        \phi(m) &= \{n \in N:\: (m,n)\in\phi\},
\end{align}
the \emph{image of $m$ by $\phi$}. Modules $M$ and $N$ are called the \emph{domain} and \emph{codomain} of $\phi$ respectively.

Given linear relations $\phi\colon M \mto N$ and $\psi\colon N \mto P$ we define the \emph{relation composition} (denoted by $\psi \circ \phi$) as follows:
    \begin{equation}
    \label{eq:composition}
        \psi \circ \phi = \{(m,p) \in M \oplus P:\: \exists{n \in N}\:\  (m,n)\in\phi,\: (n,p)\in\psi \}.
    \end{equation}
By the \emph{inverse relation} of $\phi$ we mean a relation $\phi^{-1}\colon N \mto M$ such that $(n,m)\in\phi^{-1}$ provided $(m,n)\in\phi$. Note that the inverse of a linear relation and the composition of linear relations are still linear, moreover $(\psi\circ\phi)^{-1}=\phi^{-1}\circ\psi^{-1}$ for any linear relations $\phi, \psi$. For $n \in N$ we define
    \begin{align}
        \phi^{-1}(n) &= \{m \in M:\: (m,n)\in\phi\}, 
    \end{align}
the \emph{preimage of $n$ by $\phi$}. Note that the image and preimage might be empty. Observe that taking the preimage by a relation is the same as taking the image by its inverse. Given a subset of $N$ we define its image by the union of images of its elements, similarly for preimage. If $N = M$, we will write $\phi^k$ ($k \in \N$) for $k$ compositions of $\phi$ with itself and $\phi^{-k}$ for $(\phi^{-1})^k$. 

We say that a linear relation $\phi\colon M \mto N$ is \emph{single-valued} if $\card\phi(m) \leq 1$ for $m \in M$, \emph{total} if $\card\phi(m) \geq 1$ for $m \in M$. Moreover, we say that is is \emph{injective} or \emph{surjective}, if $\phi^{-1}$ is respectively single-valued or total. Finally, we say that it is a \emph{matching} if it is single-valued and injective and \emph{correspondence} if it is total and surjective. Note that a linear relation is an module isomorphism if and only if it is a matching and a correspondence. Similarly as in case of module homomorphisms, we obtain a characterization:

\begin{prop}
\label{prop:matching characterization}
Let $\phi\colon A\mto B$ be a linear relation. Then $\phi$ is single-valued if and only if $\phi(0)=\{0\}$ and $\phi$ is injective if and only if $\phi^{-1}(0)=\{0\}$.
\end{prop}
\begin{proof}
Assume that $\phi$ is single-valued. Since always $\apply{\phi}{0}{0}$, we deduce that $\phi(0)=\{0\}$. Now if $\phi(0)=\{0\}$, we want to prove that it is single-valued. Let $x\in A$, $y,y'\in B$ be such that $\apply{\phi}{x}{y}$ and $\apply{\phi}{x}{y'}$. Then $\apply{\phi}{0}{(y-y')}$, thus $y-y'=0$ and $\phi$ is single valued. The proof of the second statement is similar.
\end{proof}

\begin{definition}
\label{defn:generalized}
    Let $\a\colon A\mto A$ be a linear relation. We define the \emph{generalized kernel of $\a$} (denoted by $\gker\a$) as 
    \begin{equation}
        \gker\a:=\left\langle\bigcup_{l\in\Z}\a^{l}(0)\right\rangle,
    \end{equation}
    where by $\langle S\rangle$ we understand a module spanned by subset $S\subseteq M$. We also define the \emph{generalized image of $\a$} (denoted by $\gim\a$) as 
    \begin{equation}
        \gim\a:=\bigcap_{l\in\Z}\a^l(A).
    \end{equation}
\end{definition} 

Since $\a^{l}(0)\subseteq\a^{l+1}(0)$ and $\a^{-l}(0)\subseteq\a^{-l-1}(0)$, observe that 
\begin{equation}
    \label{eq:generalized kernel decomposition}
    \gker\a=\bigcup_{l\in\N}\a^l(0)+\bigcup_{l\in\N}\a^{-l}(0).
\end{equation} Note that the generalized kernel and the generalized image are submodules of $A$.

\begin{prop}
    \label{prop:gim and gker inclusions}
    Let $A$ be an $R$-module and $\a\colon A\mto A$ be a linear relation. Then the folowing inclusions hold:
\begin{align*}
   \tag{a}\label{gim and gker inclusions 3}
   \a(\gker\a)&\subseteq\gker\a,\\
   \tag{b}\label{gim and gker inclusions 4}
   \a^{-1}(\gker\a)&\subseteq\gker\a.
\end{align*}
Furthermore, if $\a$ is a matching, then additionally
\begin{align*}
    \tag{c}\label{gim and gker inclusions 1}
    \gim\a&\subseteq\a(\gim\a),\\
    \tag{d}\label{gim and gker inclusions 2}
    \gim\a&\subseteq\a^{-1}(\gim\a).
\end{align*}
\end{prop}

\begin{proof}
To prove (\ref{gim and gker inclusions 3}) take $y\in\a(\gker\a)$. Let $x\in\gker\a$ be such that $\apply{\a}{x}{y}$. By (\ref{eq:generalized kernel decomposition}), there exist $n\in\a^k(0)$ and $n'\in\a^{-k'}(0)$ such that $x=n+n'$ for some $k,k'\in\N$. Since $n'\in\a^{-k'}(0)$, there exists $m\in\a^{-k'+1}(0)$ such that $\apply{\a}{n'}{m}$. From the linearity of $\a$ we obtain $\apply{\a}{n}{(y-m)}$, hence $\apply{\a^{k+1}}{0}{(y-m)}$ and consequently $y-m\in\gker\a$. Since $m\in\gker\a$, we deduce that $y\in\gker\a$. Notice that $\gker\a=\gker\a^{-1}$, thus (\ref{gim and gker inclusions 4}) follows from (\ref{gim and gker inclusions 3}) applied to $\a^{-1}$.

To prove (\ref{gim and gker inclusions 1}), assuming that $\a$ is a matching, take $x\in\gim\a$. We will construct inductively a sequence $\{x_n\}_{n\in\Z}$ such that $x_0=x$ and $\apply{\a}{x_n}{x_{n+1}}$ for $n\in\Z$. Then $x_1, x_{-1}\in\gim\a$, thus the desired equalities are satisfied. First, $x_0=x$. Assume that for $k\in\N$ we have constructed a sequence $\{x_n\}_{n=-k}^{n=k}$ such that $x_0=x$ and $\apply{\a}{x_n}{x_{n+1}}$ for $n=-k,-k+1\ldots k-1$. Since $x\in\gim\a$, there exists $y\in A$ such that $\apply{\a^{k+1}}{x}{y}$. Let $z\in A$ be such that $\apply{\a^k}{x}{z}$ and $\apply{\a}{z}{y}$. Since $\a$ is a matching, we deduce that $z=x_k$. Thus we set $x_{k+1}=y$. Constructing $x_{-k-1}$ analogously, we prove the induction step.
\end{proof}

Suprisingly, the assumption that $\a$ is a matching is essential. The counterexample, being tedious to analyze, is presented in the Appendix. Nevertheless, the second part of this proposition still remains true for $\a$ not being a matching, if we assume that $A$ is of finite length. This will be adressed in the latter sections of the paper (see Definition \ref{defn:finite length}).

%

\section{Categories and functors}
\label{sec:categories_and_functors}

Let $\C$ be a category. By $\Endo\cC$ we mean the category, whose objects are pairs $(A, \a)$, where $A\in\C$ and $\a\colon A\to A$ is an endomorphism in $\C$. The hom-set $\Endo\cC(A,B)$ consists of morphisms $\phi\in\C(A,B)$ for which the diagram
$$
\begin{diagram}
 \node{A}
 \arrow{s,l}{\varphi}
 \arrow{e,t}{\alpha}
 \node{A}
 \arrow{s,r}{\varphi}
 \\
 \node{B}
 \arrow{e,t}{ \beta}
 \node{B}
\end{diagram}
$$
commutes. Sometimes we will treat the pair $(A,\a)$ as just the endomorphism, since $A$ is determined as a source of $\a$. Notice that if $(A, \a)\in\Endo\cC$, then obviously ${\a\in\Endo\cC((A,\a), (A,\a))}$.

\begin{definition}
    Let $\C, \D$ be categories. We say that the functor $F\colon\Endo\cC\to\cD$ is \emph{normal} if $F(\a)$ is an isomorphism in $\D$ for every $(A,\a)$ in $\C$.   
\end{definition}

Normal functors were introduced by A. Szymczak in \cite{Sz1995}, but the idea had already been present in \cite{Mr1990}. They are used to establish the definition of the Conley index independent of the choice of the index pair.

\subsection{Szymczak functor}
The Szymczak functor is the universal normal functor - for any categories $\C, \D$ every normal functor $F\colon\Endo\cC\to\cD$ factors through the Szymczak functor (see Theorem \ref{thm:Szymczak functor is the universal normal functor}). A detailed construction of the Szymczak functor can be found in \cite{Sz1995}.

Given a category $\C$, in order to construct the Szymczak category $\Szym\cC$ we want to modify the $\Endo\cC$ in such a way that every endomorphism of the form $\a\colon(A,\a)\to(A,\a)$ becomes an isomorphism. Objects of $\Szym\cC$ are the objects of $\Endo\cC$. For objects $(A,\a),\ (B,\b)\in\Endo\cC$ we define a relation
\begin{equation}
    \label{eq:Szymczak equivalence relation}
    (\phi,n)\equiv(\psi,m)\Longleftrightarrow\exists k\in\NN: \phi\circ\a^{m+k}=\psi\circ\a^{n+k}
\end{equation}
for $(\phi,n), (\psi,m)\in\Endo\C((A,\a),(B,\b))\times\NN$. The fact that $\equiv$ is an equivalence relations follows from $\phi$ and $\psi$ being morphisms in $\Endo\C$. We define morphisms $\Szym\C((A,\a),(B,\b))$ to be all the equivalence classes of $\equiv$. Given morphisms ${[\phi,n]\colon (A,\a)\to(B,\b)}$ and ${[\psi,m]\colon(B,\b)\to(C,\gamma)}$ in $\Szym\cC$, we define their composition by the equation
\begin{equation}
    \label{eq:Szymczak morphism composition}
    [\psi,m]\circ[\phi,n]:=[\psi\circ\phi,m+n].
\end{equation}
The identity morphism of $(A,\a)$ is defined 
as $[\id_A,0]$. 
It is easy to check that the composition defined above is well-defined and satifies all properties of composition of morphisms. Thus $\Szym\cC$ forms a category. We also define the Szymczak functor as follows: for 
$(A,\a)\in\Endo\cC$ we set $\Szym(A,\a):=(A,\a)$. For a morphism $\phi$ in $\Endo\C$ we set $\Szym(\phi):=[\phi,0]$. Note that the inverse of $[\a,0]$ is given by $[\id_A,1]$. Indeed
$$
    [\id_A,1]\circ[\a,0]=[\a,1]=[\id_A,0],
$$
and the proof of the converse equality is similar. Hence, ${\Szym\colon\Endo\cC\to\Szym\cC}$ is a normal functor.

\begin{theorem}[{\cite[Theorem 6.1]{Sz1995}}]
    \label{thm:Szymczak functor is the universal normal functor}
    Let $\C$ be a category. The Szymczak functor ${\Szym\colon\Endo\cC\to\Szym{\C}}$ is normal. Moreover, for any category $\D$ and any normal functor ${F\colon\Endo\cC\to\D}$ there exists a unique up to equivalence functor ${F'\colon\Szym{\C}\to\D}$ such that the diagram 
    \begin{center}
    \begin{tikzcd}
        \Endo\cC \arrow[r, "F"] \arrow[d, "\Szym"'] & \cD \\
        \Szym{\cC} \arrow[ru, "\exists!\ F'"', dotted] &   
    \end{tikzcd}
    \end{center}
    commutes. 
\end{theorem}

\subsection{Leray functor}

In this part $R$ will be a commutative ring with unity. Recall that the Leray functor used in \cite{Mr1990} is a normal functor defined on the standard category of $R$-modules and their homomorphisms. Our goal is to extend its definition to $\r{R}$, that is, a category whose objects are $R$-modules and whose morphisms are linear relations with the standard composition of relations.  

For $(A,\a)\in\Endo\r{R}$ we write ${\i_A\colon\gim\a\ni a\mapsto a\in A}$ for the canonical inclusion and ${\pi_A\colon A\ni a\mapsto [a]\in A/\gker\a}$ for the canonical projection. Then we define
$$
\begin{array}{rcl}
\LE(A,\a) & := & (\gim{\a}, \i_A^{-1}\circ\a\circ \i_A), \\
\LM(A,\a) & := & (A/\gker{\a}, \pi_A\circ\a\circ\pi_A^{-1}). \\
\end{array}
$$
If $(B,\b)$ is another object in $\Endo\r{R}$ and $\phi\colon(A,\a)\mto(B,\b)$, we define ${\LE(\phi):=(\i_B^{-1}\circ\phi\circ \i_A)}$ and ${\LM(\phi):=(\pi_B\circ\phi\circ\pi_A^{-1})}$. Finally, we define the \emph{Leray functor} as the composition 
\begin{equation}
\label{eq:Leray functor}
L:=\LE\circ\LM. 
\end{equation}
Having all the necessary notions defined, we can state the main theorem of this subsection.
\begin{theorem}
    \label{thm:Leray is normal}
    Leray functor $L\colon\Endo\r{R}\to\Endo\r{R}$ is a normal functor.
\end{theorem}

\begin{prop}
    \label{prop:Leray on morphisms}
    Let $\phi\colon(A,\a)\mto(B,\b)$ be a morphism in $\Endo\r{R}$. Then $\LE(\phi)\colon \LE(A,\a)\mto\LE(B,\b)$ and $\LM(\phi)\colon\LM(A,\a)\mto\LM(B,\b)$ are morphisms in $\Endo\r{R}$.
\end{prop}

\begin{proof}
    We have to check that $\LE(\phi)$ commutes with endomorphisms of $\LE(A,\a)$ and $\LE(B,\b)$, i.e. 
$$
    (\i_B^{-1}\circ\phi\circ\i_A)\circ(\i_A^{-1}\circ\a\circ\i_A)=(\i_B^{-1}\circ\b\circ\i_B)\circ(\i_B^{-1}\circ\phi\circ\i_A).
$$
It suffices to show that the equations
\begin{align}
    \tag{a}\label{eq:Leray on morphisms helper a}(\i_B^{-1}\circ\phi\circ\i_A)\circ(\i_A^{-1}\circ\a\circ\i_A)&=\i_B^{-1}\circ\phi\circ\a\circ\i_A\\
    \tag{b}(\i_B^{-1}\circ\b\circ\i_B)\circ(\i_B^{-1}\circ\phi\circ\i_A)&=\i_B^{-1}\circ\b\circ\phi\circ\i_A
\end{align}
hold, since $\b\circ\phi=\phi\circ\a$. Since the proofs are similar, we will only prove (\ref{eq:Leray on morphisms helper a}). Notice that $\i_A\circ\i_A^{-1}\subseteq\id_A$, which proves left-to-right inclusion in the equality. In order to prove the opposite one, take $a\in\gim\a$ and $b\in\gim\b$ such that $\apply{(\i_B^{-1}\circ\phi\circ\a\circ\i_A)}{a}{b}$. There exists $c\in A$ such that $\apply{(\a\circ\i_A)}{a}{c}$ and $\apply{(\i_B^{-1}\circ\phi)}{x}{a}$. In order to finish the proof, it is sufficient to show that $x\in\gim\a$. Fix $k\in\N_+$. Since $a\in\gim\a$ and $b\in\gim\b$, there exist $a'\in A$ and $b'\in B$ such that $\apply{\a^{k-1}}{a'}{a}$ and $\apply{\b^k}{b}{b'}$. Note that $\apply{(\b^k\circ\phi)}{x}{b'}$, thus $\apply{(\phi\circ\a^k)}{x}{b'}$. Therefore, there exists $y\in A$ such that $\apply{\a^k}{x}{y}$. Thus $x\in\a^k(a')\cap\a^{-k}(y)\subseteq\a^k(A)\cap\a^{-k}(A)$. Since $k$ was arbitrary, we deduce that $x\in\gim\a$.

Similarly, in order to prove that $\LM(\phi)$ commutes with endomorphisms of $\LM(A,\a)$ and $\LM(B,\b)$, i.e
$$
    (\pi_B\circ\phi\circ\pi_A^{-1})\circ(\pi_A\circ\a\circ\pi_A^{-1})=(\pi_B\circ\b\circ\pi_B^{-1})\circ(\pi_B\circ\phi\circ\pi_A^{-1}),
$$
it suffices to show that the equations
\begin{align}
    \tag{c}\label{eq:Leray on morphisms helper c}(\pi_B\circ\phi\circ\pi_A^{-1})\circ(\pi_A\circ\a\circ\pi_A^{-1})&=\pi_B\circ\phi\circ\a\circ\pi_A^{-1}\\
    \tag{d}(\pi_B\circ\b\circ\pi_B^{-1})\circ(\pi_B\circ\phi\circ\pi_A^{-1})&=\pi_B\circ\b\circ\phi\circ\pi_A^{-1}
\end{align}
hold, since $\b\circ\phi=\phi\circ\a$. Again, since the proofs are similar, we will only prove (\ref{eq:Leray on morphisms helper c}). Notice that $\pi^{-1}_A\circ\pi_A\supseteq\id_A$, which proves the right-to-left inclusion in the equality. In order to prove the other one, take $a\in A$ and $b\in B$ such that $\apply{(\pi_B\circ\phi\circ\pi_A^{-1}\circ\pi_A\circ\a\circ\pi_A^{-1})}{[a]}{[b]}$. Without loss of generality, $a$ and $b$ are chosen in such a way that $\apply{(\phi\circ\pi_A^{-1}\circ\pi_A\circ\a)}{a}{b}$. There exist $x,\ x'\in A$ such that $\apply{\a}{a}{x}$, $\apply{(\pi_A^{-1}\circ\pi_A)}{x}{x'}$ and $\apply{\phi}{x'}{b}$. It means that $x-x'\in\gker\a$, hence there exist $n\in\a^k(0)$ and $n'\in\a^{-k}(0)$ such that $x-x'=n+n'$ for some $k\in\N_+$.  Since $\apply{\phi}{0}{0}$, we see that $\apply{(\phi\circ\a^{k'})}{n'}{0}$. However, $\phi\circ\a=\b\circ\phi$, which implies that there exists $m'\in B$ such that $\apply{\phi}{n'}{m'}$ and $\apply{\b^{k}}{m'}{0}$. It means that $m'\in\gker\b$. Let $m\in\a^{k-1}(0)$ be such that $\apply{\a}{m}{n}$. Now from the linearity of $\a$ and $\phi$ we deduce that $\apply{\a}{(a-m)}{(x-n)}$ and $\apply{\phi}{(x'+n')}{(b+m')}$. Since $x-n=x'+n'$, $[a-n]=[a]$ and $[b+m']=[b]$, the inclusion is proven.
\end{proof}

\begin{prop}
    \label{prop:Leray is a functor}
    $\LE$ and $\LM$ are endofunctors on $\Endo\r{R}$.
\end{prop}

\begin{proof}
    First, note that for any $(A,\a)\in\Endo\r{R}$ we have $\i_A^{-1}\circ\id_A\circ\i_A=\id_{\gim\a}$ and $\pi_A\circ\id_A\circ\pi_A^{-1}=\id_{A/\gker A}$, thus $\LM$ and $\LE$ preserve identities. Next, we are going to show that they agree with composition. Let $\phi\colon(A,\a)\mto(B,\b)$, $\psi\colon(B,\b)\mto(C,\c)$ be morphisms in $\Endo\r{R}$ for some $(A,\a), (B,\b), (C,\c)\in\Endo\r{R}$. We have to show that the equations
\begin{align}
    \label{eq:LE is a functor}
    (\i_C^{-1}\circ\psi\circ\i_B)\circ(\i_B^{-1}\circ\phi\circ\i_A)&=(\i_C^{-1}\circ\psi\circ\phi\circ\i_A)\\
    \label{eq:LM is a functor}
    (\pi_C\circ\psi\circ\pi_B^{-1})\circ(\pi_B\circ\phi\circ\pi_A^{-1})&=(\pi_C\circ\psi\circ\phi\circ\pi_A^{-1})
\end{align}
hold. The same argument as in the proof of Proposition \ref{prop:Leray on morphisms} shows us that the only non-trivial part in 
(\ref{eq:LE is a functor}) is the right-to-left inclusion and in (\ref{eq:LM is a functor}) is the left-to-right inclusion.

To finish the proof of (\ref{eq:LE is a functor}) let $x\in\gim\a$, $z\in\gim\c$ be such that $\apply{(\psi\circ\phi)}{x}{z}$. Let $y\in B$ be such that $\apply{\a}{x}{y}$ and $\apply{\c}{y}{z}$. We have to show that $y\in\gim\b$. Fix $k\in\N$. Since $x\in\gim\a$ and $z\in\gim\c$ there exists $x'\in A$ and $z'\in A$ such that $\apply{\a^k}{x'}{x}$ and $\apply{\b^k}{z}{z'}$. Thus $\apply{(\phi\circ\a^k)}{x'}{y}$ and $\apply{(\c^k\circ\phi)}{y}{z'}$, hence $\apply{(\b^k\circ\phi)}{x'}{y}$ and $\apply{(\psi\circ\b^k)}{y}{z'}$. Therefore, there exist $y'$ and $y''$ in $B$ such that $\apply{\b^k}{y'}{y}$ and $\apply{\b^k}{y}{y''}$, hence $y\in\b^k(y')\cap\b^{-k}(y'')\subseteq\b^k(B)\cap\b^{-k}(B)$. Since $k$ was arbitrary, we deduce that $y\in\gim\b$.

To finish the proof of equality (\ref{eq:LM is a functor}) let $x\in A$ and $z\in C$ be such that $\apply{(\pi_C\circ\psi\circ\pi_B^{-1}\circ\pi_B\circ\phi\circ\pi_A^{-1})}{[x]}{[z]}$. Without loss of generality, we may suppose that ${\apply{(\psi\circ\pi_B^{-1}\circ\pi_B\circ\phi)}{x}{z}}$. Let $y,y'\in B$ be such that $\apply{\phi}{x}{y}$, $\apply{(\pi_B^{-1}\circ\pi_B)}{y}{y'}$ and $\apply{\psi}{y'}{z}$. It means that $y-y'\in\gker\a$, i.e. there exist $n\in\b^k(0)$ and $n'\in\b^{-k'}(0)$ such that $y-y'=n+n'$ for some $k,k'\in\N$. Now we proceed as in the proof of Proposition \ref{prop:Leray on morphisms} - since $\apply{\phi}{0}{0}$, we see that $\apply{(\b^k\circ\phi)}{0}{n}$. From this and the fact that $\phi$ is a morphism in $\Endo\r{R}$ we deduce that there exists $m'\in\a^k(0)$ such that $\apply{\phi}{m}{n}$. Similarly we find $m'\in\c^{-k'}(0)$ such that $\apply{\psi}{n'}{m'}$. By linearity we obtain that $\apply{\phi}{(x-n)}{(y-n)}$ and $\apply{\psi}{(y'+n')}{(z+n')}$. Since $[x]=[x-n]$ and $[z]=[z+n']$ the inclusion is proven.
\end{proof}

Now we are ready to prove Theorem \ref{thm:Leray is normal}.
\begin{proof}
    Let $(A,\a)\in\Endo\r{R}$. We have to prove first that $(B,\b):=\LM(A,\a)$ is a matching. By Proposition \ref{prop:matching characterization} we only have to show that $\b(0)=\{0\}$ and $b^{-1}(0)=\{0\}$. Let $x\in B$ be such that $\apply{\b}{0}{x}$. Since $\b=\pi_A \circ\a\circ{\pi_A}^{-1}$, there exist $y,y'\in A$ such that $\apply{{\pi}_A^{-1}}{0}{y}, \apply{\a}{y}{y'}$ and $\apply{\pi_A}{y'}{x}$. Thus $y\in\gker\a$ and by Proposition \ref{prop:gim and gker inclusions}, $y'\in\gker\a$ as well, proving that $x=0$. The proof that $b^{-1}(0)=\{0\}$ is similar.

Next, we want to prove that $(C,\c):=\LE(B,\b)$ is a matching and a correspondence. Since $\c=\i_B^{-1}\circ\b\circ\i_B$ is a composition of matchings, it is a matching as well. Since $C=\gim\b$ and $\b$ is a matching, we can apply Proposition \ref{prop:gim and gker inclusions} -- we have $C\subseteq\b^{-1}(C)$, identifying $C$ as a subset of $B$. Thus for any $x\in C$ there exists $x'\in C$ such that $\apply{\c}{x}{x'}$, proving that $\c$ is total. The proof of surjectivity of $\c$ is similar.
\end{proof}

Note that in the proof of Theorem \ref{thm:Leray is normal} the order of compositions is crucial. Furthermore, $\LM\circ\LE$ is still a functor, but not a normal one in the fully general setting (this is shown in the Appendix). Nevertheless, after restricting to the $\Endo\fr{R}$ category, we will show that both orders of compositions yield the same result (see Corollary \ref{col:LE and LM commute}).

%

\section{Equivalence of Leray and Szymczak functors}
\label{sec:equivalence}

Let us fix some notation. If $\cC$ is a category, $A,B\in \cC$, by $A\iso{\cC} B$ we will denote the fact that they are isomorphic in $\cC$. For categories $\Endo\cC$ and $\Szym\cC$ we write shortly $A\iso{\Endo}B$ and $A\iso{\Szym} B$, if $\cC$ is clear from the context. 

Let $(A,\a), (B,\b)\in\Szym\cC$ be isomorphic. There exists mutually inverse isomorphisms $[\phi, n]\colon(A,\a)\to (B,\b)$ and $[\psi, m]\colon(B,\b)\to (A,\a)$. Note that
\begin{align}
    \label{eq:Szymczak equivalence:phi is a morphism}
    \phi\circ\a&=\b\circ\phi\\
    \label{eq:Szymczak equivalence:psi is a morphism}
    \psi\circ\b&=\a\circ\psi,
\end{align}
since $\phi$ and $\psi$ must be morphisms in $\Endo\cC$ by the definition of morphisms in $\Szym\cC$. The fact that $[\phi, n]$ and $[\psi, m]$ are mutually inverse translates to
\begin{align}
    \label{eq:Szymczak equivalence:psi phi = id}
    \psi\circ\phi\circ\a^k&=\a^{n+m+k}\\
    \label{eq:Szymczak equivalence:phi psi = id}
    \phi\circ\psi\circ\b^k&=\b^{n+m+k}
\end{align}
for some $k\in\N$. Obviously, by Theorem \ref{thm:Szymczak functor is the universal normal functor}, if $(A,\a)\iso{\Szym}(B,\b)$, then $L(A,\a)\iso{\Endo}L(B,\b)$. Unfortunately, the converse fails even for such simple objects as $\Z$-modules. 

\begin{example}
\label{ex:failure of equivalence not artinian}
    Let $R=\Z$ and consider the maps (treated as linear relations)
    ${(\cdot 2)\colon\Z\ni x\mto 2\cdot x\in\Z}$ and ${(\cdot 0)\colon\Z\ni x\mto 0\in\Z}$.
    Then $L(\Z,(\cdot 2)) = L(\Z,(\cdot 0)) = (0, \id_0)$,
    hence ${(\Z,(\cdot 2))\Leq(\Z,(\cdot 0))}$. But notice that if
    $[\phi,n]\colon(\Z,(\cdot 2))\mto(\Z,(\cdot 0))$ and
    $[\psi,n]\colon(\Z,(\cdot 0))\mto(\Z,(\cdot 2))$ were inverse morphisms in $\Szym\r{\Z}$,
    then $\phi\circ(\cdot 2)=(\cdot 0)\circ\phi$, hence $\phi(\Z)=0$.
    On the other hand, $\phi\circ\psi\circ(\cdot 2)^{k+n+m}=(\cdot 2)^k$,
    which implies that $\phi(\Z)=\Z$, a contradiction.
\end{example}

\begin{example}
    \label{ex:failure of equivalence not noetherian}
    Let $R=\Z$ and 
    $$
    {\Z(2^{\infty}):=\langle x_0,x_1,x_2\ldots|2x_0=0,2x_1=x_0,2x_2=x_1 \ldots\rangle}
    $$
    be the $2$-Pr{\"u}fer group. Consider the maps (treated as linear relations) $(\cdot 2)\colon\Z(2^{\infty})\ni x\mto 2\cdot x\in\Z(2^{\infty})$ and $(\cdot 0)\colon\Z(2^{\infty})\ni x\mto 0\in\Z(2^{\infty})$. Then again $(\Z(2^{\infty}),(\cdot 2))\Leq(\Z(2^{\infty}),(\cdot 0))$ but $(\Z(2^{\infty}),(\cdot 2))$ is not isomorphic to $(\Z(2^{\infty}),(\cdot 0))$ in Szymczak category by the similar reasoning as in Example \ref{ex:failure of equivalence not artinian}.
\end{example}

Examples \ref{ex:failure of equivalence not artinian} and \ref{ex:failure of equivalence not noetherian} suggest that we have to impose additional assumptions on the considered modules.

\begin{definition}\label{defn:finite length}
    Let $M$ be an $R$-module. We say that $M$ is \emph{Noetherian} and \emph{Artinian},
    if every increasing and decreasing sequence of its submodules stabilizes, respectively.
    A module is of \emph{finite length} if it is both Noetherian and Artinian.
    We define $\fr{R}$ to be the full subcategory of $\r{R}$, whose objects are $R$-modules of finite length.
\end{definition}

Note that objects interesting from the computational point of view, i.e. finite modules and finite-dimensional vector spaces, are of finite length. Experimental evidence (see Figure \ref{fig:z3_example}) suggests that for $(A,\a), (B,\b)\in \Endo\fr{R}$ if $L(A,\a)\iso{\Endo}L(B,\b)$, then $(A,\a)\iso{\Szym}(B,\b)$. Let us investigate the simplest case, namely the $\Z_p$ fields.

\subsection{Case of $\Z_p$ fields}
\label{sec:zp_fields}

In this section let $p$ be a prime number and $\cC$ be the full subcategory of $\fr{\Z_p}$ with only two objects, which are the zero space and the $\Z_p$ field (as a one-dimensional vector space over itself). We define linear relations $\top:=\Z_p\oplus\Z_p$, $\bot:=\{(0,0)\}$ and $(\cdot\l)\colon\Z_p\ni x\mto\l\cdot x\in\Z_p$ for $\l\in\Z_p$.

\begin{prop}
    \label{prop:Relations in Z_p field}
        There are exactly $p+3$ linear relations on $\Z_p$, that is $\top, \bot, (\cdot 0)^{-1}$ and $(\cdot\l)$ for $\l\in\Z_p$.
\end{prop}

\begin{proof}
    Let $\a$ be a linear relation on $\Z_p$. If $\a(0)=\{0\}$, then $\a$ is single-valued, since if $\apply{\a}{x}{y}$ and $\apply{\a}{x}{y'}$, then by linearity $\apply{\a}{0}{(y-y')}$. In this case if $\a\neq\bot$, then there exist $x_0,y_0\in\Z_p$ such that $\apply{\a}{x_0}{y_0}$, where $x_0\neq 0$. Hence we can define $\l:=\frac{y_0}{x_0}$. By linearity we deduce that $\apply{\a}{1}{\l}$, hence $\a$ is total as well. Thus $\a=(\cdot\l)$. Now we assume that $\a(0)\neq\{0\}$. Since $\a(0)$ is a subspace of $\Z_p$ it must be the case that $\a(0)=\Z_p$. If $\a^{-1}(0)=\{0\}$ then $\a=(\cdot 0)^{-1}$, in the other case $\a=\top$.
\end{proof}

Note that $L(\Z_p,(\cdot\l))=(\Z_p,(\cdot\l))$ for $\l\neq 0$, and $L(A,\a)=(0,\id_0)$ in other cases. The next proposition assures that all the relations of the form $(\cdot\l)$ are contained in different equivalence classes.

\begin{prop}
    \label{prop:Szymczak equivalence in auto C is trivial - Z_p case}
    Let $\l,\m\in\Z_p$. If $(\Z_p,(\cdot\l))\iso{\Szym}(\Z_p,(\cdot\m))$, then $\l=\m$.
\end{prop}
\begin{proof}
    Without loss of generality $\l\neq 0$. We will prove that in this case $\mu\neq 0$ as well. Let $[\phi,n]\colon(\Z_p,(\cdot\l))\mto(\Z_p,(\cdot\m))$ and $[\psi,m]\colon(\Z_p,(\cdot\m))\mto(\Z_p,(\cdot\l))$ be the mutually inverse isomorphisms in $\Szym\cC$. It means that there exists $k\in\N$ such that the equations
\begin{align*}
    \phi\circ(\cdot\l)&=(\cdot\m)\circ\phi\\
    \psi\circ(\cdot\m)&=(\cdot\l)\circ\psi\\
    \psi\circ\phi\circ(\cdot\l)^k&=(\cdot\l)^{n+m+k}\\
    \phi\circ\psi\circ(\cdot\m)^k&=(\cdot\m)^{n+m+k}
\end{align*}
hold. Composing the third one from the right by $(\cdot\l)^{-k}$ we obtain
$$
    \psi\circ\phi=(\cdot\l)^{n+m},
$$
since $(\cdot\l)$ is an isomorphism. Suppose now that $\mu=0$. Then
$$
\phi(\Z_p)=(\phi\circ(\cdot\l))(\Z_p)=((\cdot\m)\circ\phi)(\Z_p)\subseteq(\cdot\m)(\Z_p)=\{0\},
$$
hence $\phi(\Z_p)=\{0\}$. From Proposition \ref{prop:Relations in Z_p field} we deduce that $\phi\in\{\bot, (\cdot 0)\}$. Similarly we deduce that $\psi(\{0\})=\Z_p$, hence $\psi\in\{\top,(\cdot 0)^{-1}\}$. It implies that $\psi\circ\phi\in\{\bot, (\cdot 0), \top, (\cdot 0)^{-1}\}$. On the other hand, we have $\psi\circ\phi=(\cdot\l)^{n+m}$, a contradiction. Thus $\m\neq 0$ and $(\cdot\m)$ is an isomorphism. Therefore, by composing the fourth equality with $(\cdot\m)^{-k}$ from the right we obtain
$$
    \phi\circ\psi=(\cdot\m)^{n+m}.
$$
Since $(\cdot\m)$ is an isomorphism, we deduce that $\psi$ is a monomorphism. From the linearity of $\psi$ we deduce that $(\cdot\l)\circ\psi=\psi\circ(\cdot\l)$. Indeed,

$$
    \apply{((\cdot\l)\circ\psi)}{x}{y} \Longleftrightarrow\apply{\psi}{x}{\left(\frac{y}{\l}\right)}\Longleftrightarrow\apply{\psi}{(\l x)}{y}\Longleftrightarrow\apply{(\psi\circ(\cdot\l))}{x}{y}.
$$
Thus $\psi\circ(\cdot\l)=\psi\circ(\cdot\m)$. Finally, observe that $(\cdot\l)$ and $(\cdot\m)$ can be viewed as endomorphisms of $(\Z_p,(\cdot\l))$. Hence, $(\cdot\l)=(\cdot\m)$, therefore $\l=\m$. 
\end{proof}

To complete the description, one needs this last proposition.

\begin{prop}
    The objects $(\Z_p,\top)$, $(\Z_p,\bot)$, $(\Z_p,(\cdot 0))$, $(\Z_p,(\cdot 0)^{-1})$ and $(0,\id_0)$ are isomorphic in $\Szym\cC$
\end{prop}

\begin{proof}
    Let $\a\in\{\top,\bot,(\cdot 0), (\cdot 0)^{-1}\}$. The inverse isomorphisms between $(\Z_p,\a)$ and $(0,\id_0)$ are of the form
\begin{align*}
    [\pi\circ\a,0]\colon(\Z_p,\a)&\mto(0,\id_0)\\
    [\a\circ\pi^{-1},0]\colon(0,\id_0)&\mto(\Z_p,\a),
\end{align*}
where $\pi:\Z_p\to 0$ is the zero map. The equations
\begin{align*}
    (\pi\circ\a)\circ\a&=\id_0\circ(\pi\circ\a)\\
    (\a\circ\pi^{-1})\circ\id_0&=\a\circ(\a\circ\pi^{-1})    
\end{align*}
are satisfied since $\a\circ\a=\a$. Note that
$$
    (\pi\circ\a)\circ(\a\circ\pi^{-1})\circ\id_0=\id_0
$$
also trivially holds. To check that
$$
    (\a\circ\pi^{-1})\circ(\pi\circ\a)\circ\a=\a
$$
is satisfied, one should notice that the images and preimages of $\Z_p$ under the relations on the both sides of the equation are the same. Since $(\a\circ\pi^{-1})\circ(\pi\circ\a)$ is never a function, by Proposition \ref{prop:Relations in Z_p field} it must be an element of $\{\top,\bot,(\cdot 0), (\cdot 0)^{-1}\}$. But there are no two different relations in this set with the same image and preimage of $\Z_p$. This completes the proof.
\end{proof}

Thus we have proved Theorem \ref{thm:Leray and Szymczak equivalence} in this particular case. We illustrate the situation in the category $\C$ in case $p=3$ in Figure \ref{fig:z3_example}. 

\begin{figure}[h]%
{\includegraphics[width=0.95\textwidth]{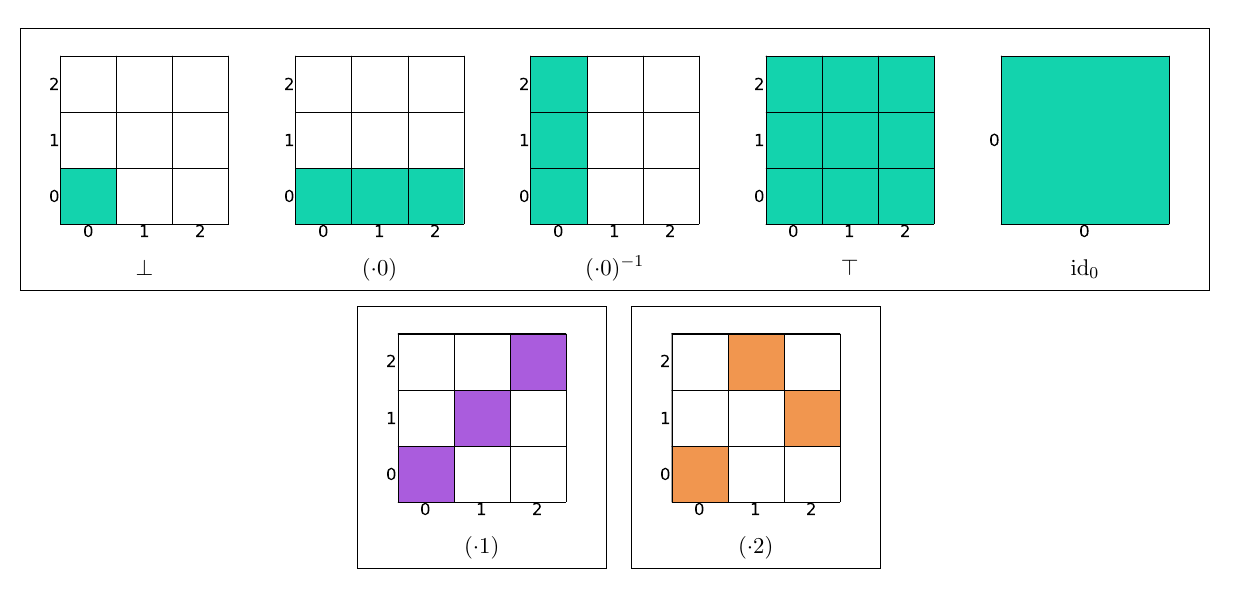}}
{\caption{Linear relations over $\Z_3$ field with $\id_0$. Every box (color) shows one class of Szym-equivalent objects.}
\label{fig:z3_example}}
\end{figure}

\subsection{General case}
\label{sec:general_case}

Before we prove the main theorem notice that the category $\fr{R}$ is closed under taking subobjects and quotients. Hence we can consider the Leray functor as a functor from $\Endo\fr{R}$ to itself. Furthermore, for $(A,\a)\in\Endo\fr{R}$ we can prove that $\gim\a\subseteq\a(\gim\a)$ and $\gim\a\subseteq\a^{-1}(\gim\a)$, without the assumption that $\a$ is a matching. Indeed, since the sequences $\{\a^l(A)\}_{l\in\N}$ and $\{\a^{-l}(A)\}_{l\in\N}$ of submodules of $A$ stabilize, there exists $k\in\N$ such that $\a^l(A)=\a^l(A)$ and $\a^{-k}(A)-\a^{-l}(A)$ for $l\geq k$.
Taking $x\in\gim\a$, there exists $y,y'\in A$ such that $\apply{\a^(k+1)}{x}{y}$ and $\apply{\a^(k-1)}{x}{y'}$. Thus there exists $x'\in A$ such that $\apply{\a}{x'}{x}$ and $\apply{\a^k}{y}{x'}$. Furthermore $\apply{\a^k}{x'}{y'}$, thus $x'\in\a^k(A)\cap\a^{-k}(A)=\gim\a$. Since $\apply{\a}{x'}{x}$, we deduce that $x\in\a(\gim\a)$, proving the first inclusion. The proof of the other one is similar.

\begin{lemma}
    \label{lemma: (A,a) is equivalent to LE(A,a)}
    Let $(A,\a)\in\Endo\fr{R}$. Then $(A,\a)\iso{\Szym}\LE(A,\a)$.
\end{lemma}

\begin{proof}
    Let $(B,\b)=\LE(A,\a)$. Since $A$ is of finite length, there exists $k\in\N$ such that $\a^k(A)=\a^l(A)$ and $\a^{-k}(A)=\a^{-l}(A)$ for all $l\geq k$. We set
\begin{align*}
    \phi&:=\i_A^{-1}\circ\a^k,\\
    \psi&:=\a^k\circ\i_A.
\end{align*}
Our goal is to prove that $[\phi,k]\colon(A,\a)\mto(B,\b)$ and $[\psi,k]\colon(B,\b)\mto(A,\a)$ are mutually inverse isomorphisms, i.e. the formulas (\ref{eq:Szymczak equivalence:phi is a morphism}) - (\ref{eq:Szymczak equivalence:phi psi = id}) are satisfied. Since $\b=\i_A^{-1}\circ\a\circ\i_A$, we can rewrite them only in terms of $\a$ and $\i_A$:
\begin{align*}
    \tag{a}\label{lemma: (A,a) is equivalent to LE(A,a) helper a}
   (\i_A^{-1}\circ\a\circ\i_A)\circ(\i_A^{-1}\circ\a^k)&=(\i_A^{-1}\circ\a^k)\circ\a\\
    \tag{b}\label{lemma: (A,a) is equivalent to LE(A,a) helper b}
    (\a^k\circ\i_A)\circ(\i_A^{-1}\circ\a\circ\i_A)&=\a\circ(\a^k\circ\i_A)\\
    \tag{c}\label{lemma: (A,a) is equivalent to LE(A,a) helper c}
    (\a^k\circ\i_A)\circ(\i_A^{-1}\circ\a^k)&=\a^{2k}\\
    \tag{d}\label{lemma: (A,a) is equivalent to LE(A,a) helper d}
    (\i_A^{-1}\circ\a\circ\i_A)^{2k}&=(\i_A^{-1}\circ\a^k)\circ(\a^k\circ\i_A).
\end{align*}
Since $\i_A\circ\i_A^{-1}\subseteq\id_A$, we see that the left-to-right inclusion is satisfied in all of the above formulas.

In order to prove the other inclusion in (\ref{lemma: (A,a) is equivalent to LE(A,a) helper a}), let $x\in A$ and $y\in B$ be such that $\apply{(\i_A^{-1}\circ\a^{k+1})}{x}{y}$. There exists $z\in A$ such that $\apply{\a^k}{x}{z}$ and $\apply{\a}{z}{y}$ (where we treat $B$ as a subset of $A$). Thus, $z\in\bigcap_{l\in\N}\a^l(A)$ from the definition of $k$ and $z\in\bigcap_{l\in\N}\a^{-l}(A)$ due to $y\in B$. Hence $z\in \bigcap_{l\in\Z}\a^l(A)=B$, which means that actually $\apply{(\i_A^{-1}\circ\a\circ\i_A\circ\i_A^{-1}\circ\a^k)}{x}{y}$.

Notice that the opposite inclusion in (\ref{lemma: (A,a) is equivalent to LE(A,a) helper b}) follows from the equality (\ref{lemma: (A,a) is equivalent to LE(A,a) helper a}) with $\a$ replaced with $\a^{-1}$ and then taking the inverse on both sides.

Now we prove right-to-left inclusion in (\ref{lemma: (A,a) is equivalent to LE(A,a) helper c}). Let $x,x'\in A$ be such that $\apply{\a^{2k}}{x}{x'}$. There exists $z\in A$ such that $\apply{\a^k}{x}{z}$ and $\apply{\a^k}{z}{x'}$, hence $z\in\a^k(A)\cap\a^{-k}(A)=B$ by the definition of $k$. Therefore $\apply{(\a^k\circ\i_A\circ\i_A^{-1}\circ\a^k)}{x}{x'}$.

It remains to prove right-to-left inclusion in (\ref{lemma: (A,a) is equivalent to LE(A,a) helper d}). Let $y,y'\in B$ be such that $\apply{(\i_A^{-1}\circ\a^{2k}\circ\i_A)}{y}{y'}$. There exists a sequence $\left\{z_l\right\}_{l=0}^{2k}$ such that $z_0=y$, $z_{2k}=y'$ and $\apply{\a}{z_l}{z_{l+1}}$ for $l=0,1,\ldots, 2k-1$. We show that $z_l\in B$ for every such $l$. Since $y\in B$, for every $l\in\N$ there exists $z_{-l}$ such that $\apply{\a^l}{z_{-l}}{y}$, and since $y'\in B$ $l\in\N$, $l>2k$, there exists $z_{l}$ such that $\apply{\a^l}{y'}{z_l}$. Thus we have obtained a sequence $\{z_l\}_{l\in\Z}$ such that $\apply{\a^j}{z_l}{z_{l+j}}$ for any $l,j\in \Z$, hence $z_l\in B$ for every $l\in\Z$.
\end{proof}

\begin{lemma}
    \label{lemma: (A,a) is equivalent to LM(A,a)}
    Let $(A,\a)\in\Endo\fr{R}$. Then $(A,\a)\iso{\Szym}\LM(A,\a)$.
\end{lemma}

\begin{proof}
    Let $(B,\b)=\LM(A,\a)$. Since $A$ is of finite length, there exists $k\in\N$ such that $\a^k(0)=\a^l(0)$ and $\a^{-k}(0)=\a^{-l}(0)$ for all $l\geq k$. We set
\begin{align*}
    \phi&:=\pi_A\circ\a^k\\
    \psi&:=\a^k\circ\pi_A^{-1}.
\end{align*}
Our goal is to prove that $[\phi,k]\colon(A,\a)\mto(B,\b)$ and $[\psi,k]\colon(B,\b)\mto(A,\a)$ are mutually inverse isomorphisms, i.e. the formulas (\ref{eq:Szymczak equivalence:phi is a morphism}) - (\ref{eq:Szymczak equivalence:phi psi = id}) hold. Since $\b=\pi_A\circ\a\circ\pi_A^{-1}$, we can rewrite them only in terms of $\a$ and $\pi_A$.
\begin{align*}
    \tag{a}\label{lemma: (A,a) is equivalent to LM(A,a) helper a}
   (\pi_A\circ\a\circ\pi_A^{-1})\circ(\pi_A\circ\a^k)&=(\pi_A\circ\a^k)\circ\a\\
    \tag{b}\label{lemma: (A,a) is equivalent to LM(A,a) helper b}
    (\a^k\circ\pi_A^{-1})\circ(\pi_A\circ\a\circ\pi_A^{-1})&=\a\circ(\a^k\circ\pi_A^{-1})\\
    \tag{c}\label{lemma: (A,a) is equivalent to LM(A,a) helper c}
    (\a^k\circ\pi_A^{-1})\circ(\pi_A\circ\a^k)&=\a^{2k}\\
    \tag{d}\label{lemma: (A,a) is equivalent to LM(A,a) helper d}
    (\pi_A\circ\a\circ\pi_A^{-1})^{2k}&=(\pi_A\circ\a^k)\circ(\a^k\circ\pi_A^{-1})
\end{align*}
Similarly as in the proof of Lemma \ref{lemma: (A,a) is equivalent to LE(A,a)}, we see that the right-to-left inclusion is satisfied in all of the formulas above since $\pi_A^{-1}\circ\pi_A\supseteq\id_A$.

In order to prove the opposite inclusion in (\ref{lemma: (A,a) is equivalent to LM(A,a) helper a}) let $x,y\in A$ be such that $\apply{(\pi_A\circ\a\circ\pi_A^{-1}\circ\pi_A\circ\a^k)}{x}{[y]}$. Without loss of generality, let $\apply{(\a\circ\pi_A^{-1}\circ\pi_A\circ\a^k)}{x}{y}$. There exist $z,z'\in A$ such that $\apply{\a^k}{x}{z}$, $\apply{(\pi_A^{-1}\circ\pi_A)}{z}{z'}$ and $\apply{\a}{z'}{y}$. Hence $z-z'\in\gker\a$. Using the definition of $k$ and (\ref{eq:generalized kernel decomposition}) we deduce that there exist $n\in\a^k(0)$ and $n'\in\a^{-k}(0)$ such that $z-z'=n+n'$. There exist $m'\in\a^{-k+1}(0)$ such that $\apply{\a}{n'}{m'}$. By linearity of $\a$ we get that $\apply{\a^k}{x}{(z-n)}$ and $\apply{\a}{(z'+n')}{(y+m')}$. Moreover, $[y]=[y+m']$, since $m'\in\gker\a$. Therefore $\apply{(\pi_A\circ\a^{k+1})}{x}{[y]}$.

Similarly as in the proof of Lemma \ref{lemma: (A,a) is equivalent to LE(A,a)}, notice that the right-to-left inclusion in (\ref{lemma: (A,a) is equivalent to LM(A,a) helper b}) follows from the equality (\ref{lemma: (A,a) is equivalent to LM(A,a) helper a}) with $\a$ replaced with $\a^{-1}$ and then taking the inverse on both sides.

Now we prove the left-to-right inclusion in (\ref{lemma: (A,a) is equivalent to LM(A,a) helper c}). Let $x,y\in A$ be such that $\apply{(\a^k\circ\pi_A^{-1}\circ\pi_A\circ\a^k)}{x}{y}$. Let $z,z'\in A$ be such that $\apply{\a^k}{x}{z}$, $\apply{(\pi_A^{-1}\circ\pi_A)}{z}{z'}$ and $\apply{\a^k}{z'}{y}$. Let $n\in\a^k(0)$ and $n'\in\a^{-k}(0)$ such that $z-z'=n+n'$. Using the linearity of $\a$ we deduce that $\apply{\a^k}{x}{(z-n)}$ and $\apply{\a^k}{(z'+n')}{y}$. Thus $\apply{\a^{2k}}{x}{y}$, since $z-n=z'+n'$.

Finally we prove the opposite inclusion in (\ref{lemma: (A,a) is equivalent to LM(A,a) helper d}).
Let $x,y\in A$ be such that $\apply{(\pi_A\circ\a\circ\pi_A^{-1})^{2k}}{[x]}{[y]}$. It means that there exists a sequence $\left\{z_l\right\}_{l=0}^{l=2k}$ such that $[z_0]=[x]$, $[z_{2k}]=[y]$ and $\apply{(\pi_A\circ\a\circ\pi_A^{-1})}{[z_l]}{[z_{l+1}]}$ for $l=0,1,\ldots, 2k-1$. Now we proceed by induction. We prove that for each $j=0,1,\ldots, 2k$ there exists a sequence $\left\{u_l\right\}_{l=0}^{2k}$ such that $[u_0]=[x]$, $[u_{2k}]=[y]$, $\apply{(\pi_A\circ\a\circ\pi_A^{-1})}{[u_l]}{[u_{l+1}]}$ for $l=0,1,\ldots, 2k-1$ and $\apply{\a}{u_l}{u_{l+1}}$ for $0\leq l<j$. Clearly, the base of the induction is satified. Now fix $j\in\{0,1,\ldots, 2k\}$ and let $\left\{u_l\right\}_{l=0}^{2k}$ be the sequence described above. Then $\apply{(\pi_A\circ\a\circ\pi_A^{-1})}{[u_j]}{[u_{j+1}]}$, that is there exists $s, t\in A$ such that $[s]=[u_j]$, $[t]=[u_{j+1}]$ and $\apply{\a}{s}{t}$. Using (\ref{eq:generalized kernel decomposition}) and the fact that $\gker\a=\a^{2k}(0)+\a^{-2k}(0)$, we deduce that there exist $n\in\a^{2k}(0)$ and $n'\in\a^{-2k}(0)$ such that $u_j-s=n+n'$. Let $m\in\a^{2k-j}(0)$ be such that $\apply{\a^j}{m}{n}$ and $m'\in\a^{-2k+1}(0)$ such that $\apply{\a}{n'}{m'}$. Notice that $\apply{\a^j}{u_0}{u_j}$, hence by linearity of $\a$ we get that $\apply{\a^j}{(u_0-m)}{(u_j-n)}$ and $\apply{\a}{(s+n')}{(t+m')}$. We define $v_0:=u_0-m$, $v_j:=u_j-n$, $v_{j+1}:=t+m'$ and $v_l:=u_l$ for $l=j+2,j+3,\ldots, 2k$. The terms $v_l$ for $l=1,\ldots, j-1$ are impilcitly defined using the formula $\apply{\a^j}{v_0}{v_j}$. Thus we have proved the inductive step. It means that for $j=2k$ we obtain a sequence $\left\{w_l\right\}_{l=0}^{2k}$ such that $[w_0]=[x]$, $[w_{2k}]=[y]$ and $\apply{\a}{w_l}{w_{l+1}}$ for $l=0,1,\ldots, 2k-1$, thus the left-to-right inclusion in (\ref{lemma: (A,a) is equivalent to LM(A,a) helper d}) is proven.
\end{proof}

Now, the proof of Theorem \ref{thm:Special Szymczak representative} is an easy consequence of Lemmas \ref{lemma: (A,a) is equivalent to LE(A,a)} and \ref{lemma: (A,a) is equivalent to LM(A,a)}. Accordingly, Theorem \ref{thm:Leray and Szymczak equivalence} follows easily.

\begin{theorem}[Theorem \ref{thm:Leray and Szymczak equivalence}]
    
    Let $(A,a),\ (B,\b)\in\Endo\fr{R}$. Then $(A,\a)\iso{\Szym} (B,\b)$ if and only if $L(A,\a)\iso{\Endo} L(B,\b)$
\end{theorem}

\begin{proof}
    Let $(A,\a),\ (B,\b)\in\Endo\fr{R}$ be such that $L(A,\a)\iso{\Endo} L(B,\b)$. In particular, it means that $L(A,\a)\iso{\Szym} L(B,\b)$. By Theorem \ref{thm:Special Szymczak representative}, we have $(A,\a)\iso{\Szym} L(A,\a)$ and $(B,\b)\iso{\Szym} L(B,\b)$. Combining those isomorphisms in $\Szym$ category yields $(A,\a)\iso{\Szym}(B,\b)$.
\end{proof}

It turns out that the object with the bijection is unique - as stated in Theorem \ref{thm: Szymczak on auto C is identity functor}
\begin{theorem}[Theorem \ref{thm: Szymczak on auto C is identity functor}]
    Let $(A,a),\ (B,\b)\in\Endo\fr{R}$ be such that $\a$ and $\b$ are bijections. If $(A,\a)\iso{\Szym} (B,\b)$ then $(A,\a)\iso{\Endo} (B,\b)$.
\end{theorem}
\begin{proof}
Let $(A,\a),\ (B,\b)\in\Endo\fr{R}$ be such that $\a$ and $\b$ are bijections and $(A,\a)\iso{\Szym}(B,\b)$. From Theorem \ref{thm:Szymczak functor is the universal normal functor} we deduce that $L(A,\a)\iso{\Endo} L(B,\b)$. But since $\a$ and $\b$ are bijections, $L(A,\a)\iso{\Endo}(A,\a)$ and $L(B,\b)\iso{\Endo}(B,\b)$ from the construction of $L$. Hence $(A,\a)\iso{\Endo}(B,\b)$.
\end{proof}

Finally, notice that in $\Endo\fr{R}$ the order of compositions of $\LE$ and $\LM$ does not matter.

\begin{corollary}
    \label{col:LE and LM commute}
    For any $(A,\a)\in \Endo\fr{R}$ we have $\LE\circ\LM(A,\a)\iso{\Endo}\LE\circ\LM(A,\a)$.
\end{corollary}

\begin{proof}
    Let $(A,\a)\in\Endo{\fr{R}}$. From Lemmas \ref{lemma: (A,a) is equivalent to LE(A,a)} and \ref{lemma: (A,a) is equivalent to LM(A,a)} we deduce that $(A,\a)\iso{\Szym}\LM\circ\LE(A,\a)$ and $(A,\a)\iso{\Szym}\LE\circ\LM(A,\a)$. Therefore, we have $\LM\circ\LE(A,\a) \iso{\Szym} \LE\circ\LM(A,\a)$. Since $\LM\circ\LE(A,\a)$ and $\LE\circ\LM(A,\a)$ are simultaneously matchings and correspondences, $\LM\circ\LE(A,\a)\iso{\Endo}\LE\circ\LM(A,\a)$ follows from Theorem \ref{thm: Szymczak on auto C is identity functor}. 
\end{proof}

%

\appendix
\section{Appendix}
\label{app:counterexample}

In this appendix we present an example demonstrating that the 
assumptions in Proposition \ref{prop:gim and gker inclusions} and
Corollary \ref{col:LE and LM commute} cannot be relaxed,
which implies that the order of compositions in 
the definition of the Leray functor 
(see (\ref{eq:Leray functor})) cannot be reversed.

\begin{figure}[h]
\begin{tikzcd}
                     &                      &                                      & e_0 \arrow[rd, shift left=3] \arrow[rdd, shift left=2] \arrow[rddd, shift left] \arrow[rdddd] &                     &                     &           \\
                     &                      & {e_{1,-1}} \arrow[ru, shift left=2]  &                                                                                               & {e_{1,1}}           &                     &           \\
                     & {e_{2,-2}} \arrow[r] & {e_{2,-1}} \arrow[ruu, shift left=2] &                                                                                               & {e_{2,1}} \arrow[r] & {e_{2,2}}           &           \\
{e_{3,-3}} \arrow[r] & {e_{3,-2}} \arrow[r] & {e_{3,-1}} \arrow[ruuu, shift left]  &                                                                                               & {e_{3,1}} \arrow[r] & {e_{3,2}} \arrow[r] & {e_{3,3}} \\
                     &                      & \vdots \arrow[ruuuu]                 &                                                                                               & \vdots              &                     &          
\end{tikzcd}
\caption{Schematic depiction of the relation $\beta$ considered in the Appendix}
\label{fig:spider}
\end{figure}
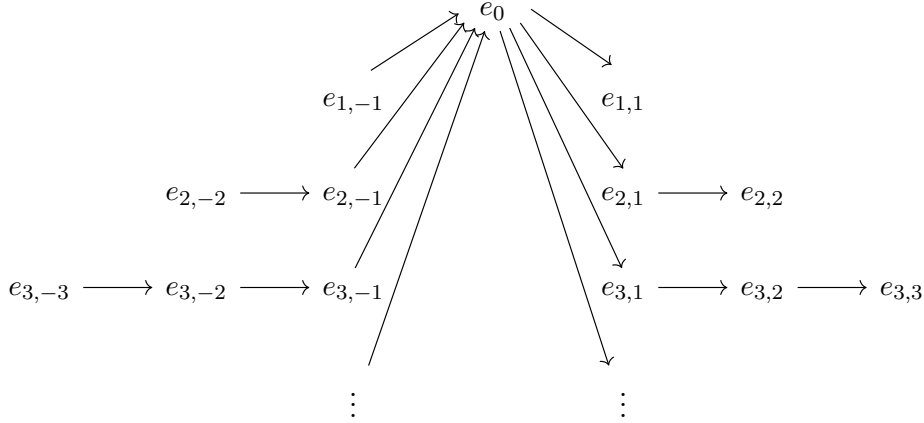

For $k\in\N_+$ let $E_k:=\{e_{k,s}\colon s=-k, -k+1, \ldots, k-1\}$,
where $e_{k,0}=e_{k',0}$ for $k,k'\in\N_+$.
Therefore $e_0:=e_{k,0}$ (where $k\in\N_+$ is arbitrary) is well-defined.
We want to think of $E_k$ as an orbit of length $2k$.
Let $\b_k$ be a relation on $E_k$ defined as 
$\apply{\b}{e_{k, s}}{e_{k, s+1}}$ for $-k\leq s < k$.
Thus $\b:=\bigsqcup_{k\in\N_+} \b_k$ is a relation (not yet linear)
on $E:=\bigsqcup_{k\in\N_+} E_k$.
See Figure \ref{fig:spider} for a schematic depiction of $\b$.
Observe that $\b^i(e_0)\neq \emptyset$ for $i\in \mathbb{Z}$,
however there is no infinite sequence $\{x_i\}_{i\in\mathbb{Z}}$
such that $\apply{\b}{x_i}{x_{i+1}}$ and $x_0=e_0$ 
(note that existence of such sequence is crucial in the proof of 
Proposition \ref{prop:gim and gker inclusions}).
In order to satisfy the formal constraints, we need to consider
a linear relation. 
Let $A$ be a free $\mathbb{Z}_2$-vector space with basis $E$
and let 
\begin{equation}
  \a:=\generatedby{\b}\subseteq A\oplus A
  \label{eq:alpha definition}
\end{equation}
be a subspace 
generated by the subset $\b$. 
For $k\in\N$ let $\N_k$ be the set of natural numbers 
greater or equal than $k$.

\begin{lemma}
  \label{couterexample:lem:determining-a(0)}
  For $\a$ defined in (\ref{eq:alpha definition})
  the following equalities
  \begin{align*}
    \a(0)=
    &\left\{
        \sum_{s\in S}e_{s,1}\colon S\subseteq\N_1
        \card{S}\: \text{is even}
    \right\},\\
    \a^{-1}(0)=
    &\left\{
        \sum_{s\in S}e_{s,-1}\colon S\subseteq\N_1,
        \card{S}\: \text{is even}
    \right\}
  \end{align*}
  hold.
\end{lemma}

\begin{proof}
    The right-to-left inclusion in the first equality is obvious.
    To prove the opposite one, let $w\in\a(0)$. 
    Then $w=\sum_{i\in I}f_i$, where $f_i\in E$. 
    Without loss of generality $I$ is such that if $f_{i_0}=f_{i_1}$
    then $i_0=i_1$.  
    We proceed by induction on $\card{I}$.
    If $\card{I}=0$, then $w=0$ and the lemma is satisfied.

    For the induction step assume that the lemma is true
    for all $w'=\sum_{i\in I'}f_{i'}'$ such that $\apply{\a}{0}{w'}$
    and $\card{I'}<\card{I}$.
    Note that ${0\oplus w=\sum_{j\in J} g_j\oplus h_j}$,
    where $g_j, h_j\in E$ and $\apply{\b}{g_j}{h_j}$,
    since $\apply{\a}{0}{w}$.
    Similarly as above, $J$ can be reduced to the case that
    if $g_{j_0}\oplus h_{j_0}=g_{j_1}\oplus h_{j_1}$, then $j_0=j_1$.
    However, it is still possible that $h_{j_0}=h_{j_1}$ for $j_0 \neq j_1$.
    For $i\in I$ define
    $$
        J_i:=\{j\in J\colon h_j=f_i\}
    $$
    and $J'=J\setminus\bigcup_{i\in I} J_i$.
    Note that sets $J_i$ are mutually disjoint.
    If $j\in J_{i_0}\cap J_{i_1}$, then $h_j=f_{i_0}=f_{i_1}$,
    thus $i_0=i_1$ by the assumption on $I$. 
    Consequently, we can write $J=J'\sqcup\bigsqcup_{i\in I} J_i$, hence

    $$
        \sum_{i\in I}f_i=
        \sum_{j\in J}h_j=
        \left(\sum_{i\in I}\sum_{j\in J_i}h_j\right)+
        \sum_{j\in J'}h_j=\sum_{i\in I}\left(\card{J_i}\right)f_i+
        \sum_{j\in J'}h_j.
    $$
    Since $f_i$ and $h_j$ are linearly independent,
    we deduce that $\card{J_i}$ is odd for all $i\in I$ 
    and that $\sum_{j\in J'}h_j=0$. 
    Now we claim that actually $J'=\emptyset$, 
    thus ${J=\bigsqcup_{i\in I}J_i}$. 
    If $j_0\in J'$, there exists $j_1\in J'$ such that $j_0\neq j_1$ 
    but $h_{j_0}=h_{j_1}$ since $\sum_{j\in J'}h_j=0$. 
    By the assumption on $J$ we deduce that $g_{j_0}\neq g_{j_1}$.
    This yields four formulae:

    \begin{align*}
        \begin{cases}
            h_{j_0}=h_{j_1},\\
            g_{j_0}\neq g_{j_1},\\
            \apply{\b}{g_{j_0}}{h_{j_0}},\\
            \apply{\b}{g_{j_1}}{h_{j_1}}.
        \end{cases}
    \end{align*}
    They can be satisfied only if 
    $h_{j_0}=h_{j_1}=e_0$ and $g_{j_0}=e_{s_0, -1}$,
    $g_{j_1}=e_{s_1, -1}$ for some $s_0, s_1 \in\N_1$.
    Since $\sum_{j\in J}g_j=0$, we deduce that 
    there exists $j_2\in J$ such that $j_0\neq j_2$
    but $g_{j_2}=g_{j_0}$.
    Since $g_{j_0}=e_{s_0, -1}$, the only choice 
    for $h_{j_2}$ is $e_0$. 
    This contradicts the assumption on $J$.
    Thus $J'=\emptyset$.

    To finish the proof we use induction hypothesis. 
    Using the fact that $I\neq\emptyset$, let $i_0\in I$. 
    Since $\card{J_{i_0}}$ is odd, we deduce that $J_{i_0}\neq\emptyset$.
    Let $j_0\in J_{i_0}$. Again, since $\sum_{j\in J}g_j=0$, 
    there exists $j_1\in J$ such that $j_0\neq j_1$ and $g_{j_0}=g_{j_1}$.
    If $j_1\in J_{i_0}$, then $h_{j_0}=h_{j_1}=f_{i_0}$,
    violating the assumption on $J$.
    Thus there exists $i_1\in I$ such that $i_0\neq i_1$ and $j_1\in J_{i_1}$,
    since $J=\bigsqcup_{i\in I} J_i$. The formulae

    \begin{align*}
        \begin{cases}
            h_{j_0}\neq h_{j_1},\\
            g_{j_0}=g_{j_1},\\
            \apply{\b}{g_{j_0}}{h_{j_0}},\\
            \apply{\b}{g_{j_1}}{h_{j_1}}  
        \end{cases}
    \end{align*}
    can be satisfied only if $g_{j_0}=g_{j_1}=e_0$, $h_{j_0}=e_{s_0, 1}$
    and $h_{j_1}=e_{s_1, 1}$ for some $s_0,s_1\in\N_+$.
    However, for 
    $
      w':=w+h_{j_0}+h_{j_1}=w+f_{i_0}+f_{i_1}=
        \sum_{i\in I\setminus\{i_0, i_1\}} f_i
    $ we have

    $$
        0\oplus w'= \sum_{j\in J\setminus\{j_0, j_1\}} g_j\oplus h_j,
    $$
    since $g_{j_0}+g_{j_1}=0$ and $(g_{j_0}+g_{j_1})\a (h_{j_0}+h_{j_1})$. 
    Thus by the induction step the first equality is proven. 
    Proof of the second one is analogous.
\end{proof}

\begin{lemma}
    \label{couterexample:lem:determining-ak(0)}
    For $\a$ defined in (\ref{eq:alpha definition}) and $k\in\N_+$ 
    the following equalities
    \begin{align*}
        \a^k(0)=
        &{\left\{
            \sum_{i=1}^k \sum_{s_i\in S_i} e_{s_i, i}\colon
            \text{$\card{S_i}$ is even for all $i=1,\ldots, k$}
        \right\}},\\
        \a^{-k}(0)=
        &{\left\{
            \sum_{i=1}^k \sum_{s_i\in S_i} e_{s_i, -i}\colon
            \text{$\card{S_i}$ is even for all $i=1,\ldots,k$}
        \right\}}
    \end{align*}
    hold.
\end{lemma}

\begin{proof}

  Again, since proofs are similar, we will focus only on the first equality.
  The right-to-left inclusion is obvious.
  To prove the opposite one, we proceed by induction on $k$.
  The base case $k=1$ is precisely 
  Lemma \ref{couterexample:lem:determining-a(0)}.
  Assume that the lemma is satisfied for some $k\in\N_+$.
  If $w\in\a^{k+1}(0)$, then there exists $v\in\a^k(0)$ 
  such that $\apply{\a}{v}{w}$.
  By the induction hypothesis,
  $v=\sum_{i=1}^k \sum_{s_i\in S_i} e_{s_i, i}$,
  where $S_i\subseteq\N_i$ is of even cardinality for all $i=1,\ldots,k$.
  Suppose for a while that $i<s_i$ for $i=1,\ldots,k$ and $s_i\in S_i$.
  Then we can define $w'=\sum_{i=1}^k \sum_{s_i\in S_i} e_{s_i, i+1}$.
  Since for every $i=1,\ldots,k$ and $s_i\in S_i$ we have
  $\apply{\b}{e_{s_i, i}}{e_{s_i,i+1}}$, we deduce that $\apply{\a}{v}{w'}$.
  It means that $\apply{\a}{0}{(w+w')}$ and
  by the Lemma \ref{couterexample:lem:determining-a(0)} we deduce that
  $w+w'=\sum_{t\in T}e_{t,1}$ for some $T\subseteq \N_1$ of
  even cardinality.
  Thus
  $$
      w=\sum_{i=1}^k \sum_{s_i\in S_i} e_{s_1, i+1}+
      \sum_{t\in T}e_{t,1},
  $$
  proving the induction step.

  It remains to show that actually $i<s_i$ for $i=1,\ldots,k$ and $s_i\in S_i$.
  Since $\apply{\a}{v}{w}$,
  $$
      v\oplus w=
      \sum_{j\in J}g_j\oplus h_j,
  $$
  where $g_j, h_j \in E$ and $\apply{\b}{g_j}{h_j}$ for $j\in J$. Note that for every $i=1,\ldots,k$ and $s_i\in S_i$ there exists $j_{s_i}$ such that $e_{s_i, i}=g_{j_{s_i}}$ since
  $$
      \sum_{i=1}^k \sum_{s_i\in S_i} e_{s_i, i}=\sum_{j\in J}g_j.
  $$
  Thus for every such $i$ and $s_i$ there exists $h_{j_{s_{i}}}$ such that $\apply{\b}{e_{s_i,i}}{h_{j_{s_i}}}$, 
  which is possible only~if~$i<s_i$.
\end{proof}

\begin{lemma}
    \label{counterexample:lem:determining-image-by-representative}
    For $\a$ defined in (\ref{eq:alpha definition}), 
    $k\in\Z$ and $x,y\in A$,
    if $\apply{\a^k}{x}{y}$, then $\a^k(x)=y+\a^k(0)$.
\end{lemma}

\begin{proof}
    To prove left-to-right inclusion, let $z\in\a^k(x)$.
    Then $\apply{\a^k}{x}{z}$, thus by linearity $\apply{\a^k}{0}{(z-y)}$,
    hence $z-y\in\a^k(0)$.
    Consequently $z=y+(z-y)\in y+\a^k(0)$.
    To prove the opposite inclusion, let $n\in\a^k(0)$. 
    By linearity $\apply{\a^k}{x}{(y+n)}$, therefore $y+n\in\a^k(x)$.
\end{proof}

\begin{lemma}
  \label{counterexample:determining-gim}
  For $\a$ defined in (\ref{eq:alpha definition}) we have $\gim\a=\{0, e_0\}$.
\end{lemma}

\begin{proof}
    The right-to-left inclusion is obvious. 
    To prove the opposite, let $w\in\gim\a$. Then
    $$
        w=\lambda e_0+
        \sum_{i\in\mathbb{Z}\setminus\{0\}}\sum_{r_i\in R_i} \mu_{r_i,i} e_{r_i,i},
    $$
    where $R_i\subseteq\N_i$, $\lambda, \mu_{r_i,i} \in\Z_2$ 
    and only finitely many of $\mu_{r_i, i}$ are non-zero.
    Since ${e_0\in\gim\a}$, without loss of generality $\lambda=0$. 
    Indeed, if $w\in\gim\a$, then $w+\lambda e_0\in\gim\a$ as well, but
    $$
        w+\lambda e_0=\sum_{i\in\mathbb{Z}\setminus\{0\}}\sum_{r_i\in R_i} \mu_{r_i,i} e_{r_i,i}.
    $$
    We claim that in fact $\mu_{r_i, i}=0$ 
    for all the possible choices of $i$.
    Since the proofs are analogous, we consider only the case $i > 0$.

    Assume that there exists some $\mu_{r_i,i}\neq 0$ with $i > 0$.
    Then $k:=\min\{r_i - i\colon {i\in\N_1, r_i\in R_i}\}$ 
    is well-defined and
    nonnegative, since $r_i\in R_i\subseteq \N_i$.
    We claim that $\a^{k+1}(w)=\emptyset$, thus ${w\notin\gim\a}$. 
    Let $i_0$ be such that $k=r_{i_0}-i_0$ 
    for some $r_{i_0}\in R_{i_0}$.
    We will show that each $u\in\a^k(w)$ is of the form
    $u=e_{r_{i_0},r_{i_0}}+u'$, where $u'$ is a linear combination of
    elements of $E$ different than $e_{r_{i_0},r_{i_0}}$. 
    This however will imply that $\a(u)=\emptyset$.
    Indeed, if $\apply{\a}{u}{v}$ for some $v\in A$, then 
    $$
      u\oplus v = \sum_{j\in J} g_j \oplus h_j
    $$
    for some $g_j, h_j\in E$ such that $\apply{\b}{g_j}{h_j}$.
    Since $u=e_{r_{i_0}, r_{i_0}}+u'$, there exists $j_0\in J$
    such that $e_{r_{i_0}, r_{i_0}}=g_{j_0}$. 
    However, it implies that 
    $\apply{\b}{e_{r_{i_0}, r_{i_0}}}{h_{j_0}}$, which contradicts
    the definition of $\b$.

    It remains to prove that for every $u\in\a^k(w)$ 
    such $u'$ exists.
    First, note that for $i < 0$ and $r_i \in R_i$ there 
    exists $f_{r_i, i} \in E \setminus \{e_{i_0,i_0}\}$
    such that $\apply{\b^k}{e_{r_i, i}}{f_{r_i, i}}$
    (roughly speaking, such $f_{r_i, i}$ can be found by 
    iterating $e_{r_i, i}$
    until $e_0$ and afterwards choosing $e_{l,1}$ with
    sufficiently big $l$ if necessary).
    Next, for $i > 0$ and $r_i\in R_i$ notice that 
    $e_{r_i, i+k}$ exists by the definition of $k$.
    Since $\apply{\b^k}{e_{r_i,i}}{e_{r_i,i+k}}$, 
    we deduce that $\apply{\a^k}{w}{v}$, where
    $$
        v = 
        \sum_{i=1}^{+\infty}\sum_{r_i\in R_i} \mu_{r_i,i} e_{r_i,i+k} +
        \sum_{i=-1}^{-\infty}\sum_{r_i\in R_i} \mu_{r_i,i} f_{r_i,i}.
    $$
    Thus $v=e_{r_{i_0},r_{i_0}}+v'$, where $v'$ is a 
    linear combination of elements of $E$ 
    different than $e_{r_{i_0},r_{i_0}}$.

    From the Lemma
    \ref{counterexample:lem:determining-image-by-representative}
    we know that if $u\in\a^k(w)$, then $u=e_{i_0,i_0}+v'+n$,
    where $n\in\a^k(0)$. 
    Using Lemma \ref{couterexample:lem:determining-ak(0)},
    we deduce that $n=\sum_{i=1}^k\sum_{s_i\in S_i}e_{s_i,i}$
    (where $S_i\subset \N_i$ is of even cardinality). 
    Note that $r_{i_0}>k$, therefore $n$ is a linear combination
    of elements of $E$ different than $e_{r_{i_0},r_{i_0}}$.
    For $u':=v'+n$, we have $u=e_{r_{i_0},r_{i_0}}+u'$,
    finishing the proof.
\end{proof}

Using Lemma \ref{counterexample:determining-gim},
we see that $\gim\a\not\subseteq\a(\gim\a)$ 
and $\gim\a\not\subseteq\a^{-1}(\gim\a)$.
Therefore assumption that a linear relation
is a matching in Proposition \ref{prop:gim and gker inclusions}
cannot be relaxed.

Finally, we will prove that 
$\LE \circ \LM \neq \LM \circ \LE$.
Since $\LE(A, \a)=(\Z_2, \bot)$, we deduce that
$\LM \circ \LE(A,\a)=(\Z_2, \bot)$ as well. 
From the construction of $\LE$ and $\LM$ we get that
$\LM\circ\LE(\a)=\bot$ as well 
(here $\a$ is treated as an endmorphism of $(A,\a)$ in $\Endo\r{\Z_2}$).
However, $\bot\colon(\Z_2, \bot) \mto (\Z_2, \bot)$ is not an isomorphism.
Since from Theorem \ref{thm:Leray is normal} we know that 
$\LE \circ \LM(\a)$ must be an isomorphism, the proof is finished.



\begin{thebibliography}{99}

\bibitem{AMPW2024}
E. Akin, M. Mrozek, M. Przybylski and J. Wiseman, 
A complete invariant for shift equivalence for Boolean matrices and finite relations,
\textit{Topology and its Applications}, 357 (2024), 109075.

\bibitem{BMW2024}
J. Barmak, M. Mrozek and T. Wanner, Conley 
Index for Multivalued Maps on Finite Topological Spaces,
\textit{Foundations of Computational Mathematics}, (2024), \textit{Online First}.

\bibitem{BGHKMV2024}
B. Batko, M. Gameiro, Y. Hung, W. Kalies, K. Mischaikow and E. Vieira,
Identifying nonlinear dynamics with high confidence from sparse data,
\textit{SIAM Journal on Applied Dynamical Systems}, 23 (2024), 1, {383}--{409}.

\bibitem{BMMP2020}
B. Batko, K. Mischaikow, M. Mrozek and M. Przybylski,
Conley index approach to sampled dynamics,
\textit{SIAM Journal on Applied Dynamical Systems}, 19 (2020), 1, {665}--{704}.

\bibitem{CC2017}
C-C. Chen, J. A. Conejero, M. Kosti{\'c} and M. Murillo-Arcila,
Dynamics of multivalued linear operators,
\textit{Open Mathematics} 15 (2017), 1, {948}--{958}.

\bibitem{Con1978}
C. Conley, 
Isolated invariant sets and the Morse index,
\textit{Regional conference series in mathematics}, 38 (1978), AMS.

\bibitem{FR2000}
J. Franks and D. Richeson,
Shift equivalence and the Conley index,
\textit{Transactions of the American Mathematical Society}, 352 (2000), 7, {3305}--{3322}.

\bibitem{J2020}
M. Juda,
Unsupervised features learning for sampled vector fields,
\textit{SIAM Journal on Applied Dynamical Systems}, 19 (2020), 4, {2720}--{2736}.

\bibitem{KM1995}
T. Kaczynski and M. Mrozek, 
Conley index for discrete multi-valued dynamical systems,
\textit{Topology and its Applications}, 65 (1995), 1, {83}--{96}.

\bibitem{ML1995}
S. Mac Lane, 
Homology,
\textit{Classics in Mathematics}, (1995), Springer Berlin, Heidelberg.

\bibitem{M1966}
R.E. Moore, 
Interval analysis,
\textit{Series in automatic computation}, (1966), Prentice-Hall. 

\bibitem{Mr1990}
M. Mrozek,
Leray functor and cohomological Conley index for discrete dynamical systems,
\textit{Transactions of the American Mathematical Society}, 318 (1990), 1, {149}--{178}.

\bibitem{Sz1995}
A. Szymczak, 
The Conley index for discrete semidynamical systems,
\textit{Topology and its Applications}, 66 (1995), 3, {215}--{240}.

\bibitem{Mr1999}
M. Mrozek,
Construction and properties of the Conley index,
\textit{Banach Center Publications}, 47 (1999), 1, {29}--{40}.

\bibitem{PMW2023}
M. Przybylski, M. Mrozek and J. Wiseman,
The Szymczak Functor and Shift Equivalence on the Category of Finite Sets and Finite Relations,
\textit{Journal of Dynamics and Differential Equations}, 37 (2025), {1835}--{1870}.

\bibitem{RS1988}
J. W. Robbin and D. Salamon,
Dynamical systems, shape theory and the Conley index,
\textit{Ergodic Theory \& Dynamical Systems}, 8 (1988), {375}--{393}.

\bibitem{V1927}
L. Vietoris,
{\"U}ber den h{\"o}heren Zusammenhang kompakter R{\"a}ume und eine Klasse von zusammenhangstreuen Abbildungen,
\textit{Mathematische Annalen}, 97 (1927), 1, {454}--{472}.











\end{thebibliography}
\end{document}